\documentclass[11pt]{article}

\usepackage[T1]{fontenc}
\usepackage{upgreek,euscript}
\usepackage[utf8]{inputenc}
\usepackage{microtype}
\usepackage{hyperref}

\usepackage{graphicx}
\usepackage{mathtools, amsmath, amssymb, amsthm, enumerate, faktor,float}
\usepackage{cleveref}
\usepackage[paper=a4paper,                     
            left=27mm,                         
            right=27mm,                        
            top=30mm,                          
            bottom=25mm,                       
            bindingoffset=0mm                  
           ]{geometry}
\usepackage{xcolor}
\usepackage{tikz-cd}
\usepackage{tikz}
\usepackage{amsthm}
\usepackage{amsmath,amssymb}
\usetikzlibrary{trees,positioning,arrows.meta}
\usepackage[backend=biber, style=alphabetic]{biblatex}
\usepackage{hyperref}
\allowdisplaybreaks

\numberwithin{equation}{section}
\theoremstyle{plain}
\newtheorem{definition}{Definition}[section]
\newtheorem{theorem}{Theorem}[section]
\newtheorem{lemma}[theorem]{Lemma}
\newtheorem{corollary}[theorem]{Corollary}
\newtheorem{proposition}[theorem]{Proposition}

\newtheorem{assumption}[theorem]{Assumption}

\theoremstyle{definition}
\newtheorem{remark}[theorem]{Remark}

\newcommand{\QQ}{\mathbb{Q}}
\newcommand{\GG}{\mathbb{G}}
\newcommand{\PP}{\mathbb{P}}
\newcommand{\CC}{\mathbb{C}}
\newcommand{\ZZ}{\mathbb{Z}}

\newcommand{\cM}{\mathcal{M}}
\newcommand{\cO}{\mathcal{O}}

\newcommand{\cL}{\mathcal L}

\newcommand{\cY}{\mathcal Y}
\newcommand{\cA}{\mathcal A}
\newcommand{\cX}{\mathcal X}
\newcommand{\cF}{\mathcal F}

\newcommand{\cK}{\mathcal K}

\renewcommand{\AA}{\mathbb{A}}
\newcommand{\sing}{\mathrm{sing}}

\newcommand{\sm}{\mathrm{sm}}
\newcommand{\QS}{\mathrm{QS}}

\DeclareMathOperator{\PGL}{PGL}

\DeclareMathOperator{\Aut}{Aut}
\DeclareMathOperator{\Stab}{Stab}

\DeclareMathOperator{\Cl}{Cl}

\DeclareMathOperator{\Hom}{Hom}

\DeclareMathOperator{\Pic}{Pic}

\DeclareMathOperator{\pr}{pr}
\DeclareMathOperator{\Char}{Char}
\DeclareMathOperator{\Qcoh}{Qcoh}
\DeclareMathOperator{\Coh}{Coh}
\DeclareMathOperator{\Bl}{Bl}
\DeclareMathOperator{\codim}{codim}
\DeclareMathOperator{\Lie}{Lie}

\DeclareMathOperator{\Low}{\mathbf{L}}
\DeclareMathOperator{\Mid}{\mathbf{M}}
\DeclareMathOperator{\Up}{\mathbf{U}}

\newcommand{\CN}{\textcolor{red}}
\newcommand{\DB}{\textcolor{blue}}

\title{Finite automorphism groups of hypersurfaces via Bott vanishing}

\author{Dominic Bunnett \& Caroline Namanya}

\date{}

\begin{document}

\maketitle

\begin{abstract}
We give a cohomological criterion for the automorphism group scheme of a hypersurface in a smooth Deligne–Mumford stack to be discrete in tame characteristics.
The criterion converts Bott-type vanishing on the ambient stack into vanishing of infinitesimal automorphisms of every quasismooth hypersurface in a linear system.
We apply it to toric orbifolds, obtain uniform finiteness results for quasismooth weighted-projective hypersurfaces in tame characteristic and hypersurfaces in products of projective spaces, and identify the precise exceptional windows in which the method fails.
\end{abstract}

\section{Introduction}

A classical problem in the geometry of hypersurfaces is to understand when their automorphism groups are finite. For a projective scheme \(X\) over a field \(k\), the Lie algebra of the automorphism group scheme is \[ \Lie(\Aut_{X/k}) \cong H^0(X,T_X). \] Thus the vanishing of \(H^0(X,T_X)\) implies that \(\Aut_{X/k}\) is discrete, by which we mean \'etale over \(k\).
 When the automorphism group scheme is of finite type, this also gives finiteness.
Finiteness is the property we are most concerned with in this paper. However, discreteness is the natural statement in full generality: in the Calabi--Yau case a smooth variety can have a discrete but infinite automorphism group, as already occurs for smooth quartic K3 surfaces. We therefore formulate our general results in terms of discreteness, and obtain finiteness where additional geometry supplies it.

The general type case is classical: by a theorem of Matsumura \cite{Matsumura:1963}, the automorphism group scheme of a variety of general type is finite although may be non-reduced in characteristic $p>0$.
Our main interest is therefore in Fano and Calabi--Yau hypersurfaces.
For hypersurfaces in projective space, there are two classical approaches to the vanishing of \(H^0(X,T_X)\).
Matsumura and Monsky \cite{MatsumuraMonsky:1963} use the conormal sequence and the Euler sequence on projective space.
Their argument uses the Euler relation, which introduces restrictions on the characteristic; see \cite[Corollary 3.5]{Huybrechts:2023} for a short account.
By contrast, the earlier method of Kodaira and Spencer \cite{KodairaSpencer:1958} proceeds through restriction and Poincar\'e residue sequences for differential forms.
Kodaira and Spencer were concerned with the non-existence of non-zero global vector fields rather than with automorphism groups themselves.

It is this second method that we develop here. Katz and Sarnak \cite[Chapter 11.6]{KatzSarnak:1999}, following an observation attributed there to Gabber, noted that the Kodaira--Spencer argument is algebraic and works in arbitrary characteristic.
The essential input is the vanishing of certain cohomology groups of twisted differential forms on the ambient projective space.
For a smooth hypersurface \(X\subset Y\), the restriction and Poincar\'e residue sequences reduce the vanishing of \(H^0(X,T_X)\) to a finite collection of cohomological vanishings on \(Y\). These involve twisted differential forms, precisely the setting of Bott-type vanishing.
The same argument applies to smooth Cartier divisors in smooth Deligne--Mumford stacks.
The hypotheses depend only on the ambient space and the hypersurface class, so they give a uniform statement for every smooth member of the corresponding linear system.
Our main criterion also allows us to pass to the coarse moduli space.

\begin{theorem}[Discreteness criterion] \label{theorem:main-theorem}
Let \(k\) be a field and let \(\cY\) be a smooth, separated Deligne--Mumford stack over \(k\) of dimension \(n\geq 2\).
Let \(i:\cX\hookrightarrow\cY\) be a smooth effective Cartier divisor.
For \(q\geq 0\), write \[\cL_q \coloneqq \omega_{\cY}^{\otimes -1}\otimes \cO_{\cY}(-(q+1)\cX).\]
Let \(\pi_X:\cX\to X\) be the coarse moduli space. 
Suppose that \(\pi_X\) is an isomorphism outside a closed subset of \(X\) of codimension at least \(2\) and that for every \(q\geq0\) we have
\[H^q\left(\cY,\Omega_{\cY}^{n-q-1}\otimes\cL_q\right)=H^q\left(\cY,\Omega_{\cY}^{n-q-2}\otimes\cL_q\right)=0,\] where \(\Omega_{\cY}^r=0\) for \(r<0\).
Then
\[H^0(X,T_X)=0.\]
\end{theorem}

Bott vanishing has received renewed attention in recent years, with results for Fano threefolds, certain GIT quotients, and varieties admitting suitable endomorphisms \cite{Totaro:2024a,Torres:2025,KawakamiTotaro:2025}.
We extend Bott vanishing to the tame canonical stacks of projective simplicial toric varieties by passing to their coarse spaces and applying vanishing for ample Weil divisors, following \cite{Bott:1957}, \cite{Danilov:1978}, \cite{Steenbrink:1993}    and \cite{Fujino:2007} for cases over any field.
We also use the vanishing of untwisted differential forms off the diagonal,
\[
H^q(\cY,\Omega_{\cY}^p)=0 \qquad \text{for }p\neq q.
\]
We call Bott vanishing together with this additional condition \emph{extended Bott vanishing}.
These results allow us to check many of the hypotheses of our criterion.
The recent examples beyond toric varieties suggest further applications of the Kodaira--Spencer method, provided the additional vanishings required by the criterion can also be established.

Our principal motivation comes from Fano hypersurfaces. Their automorphism groups are of finite type because automorphisms preserve the ample anticanonical class, so discreteness implies finiteness. Our method therefore gives uniform finiteness results for Fano hypersurfaces in a given linear system.

We were also motivated by K-stability: a K-stable Fano variety has finite automorphism group, so finiteness provides a basic necessary check for K-stability. More broadly, this raises the question of whether smooth or quasismooth Fano hypersurfaces of natural degree in toric varieties can nevertheless be K-unstable. Recent work of \cite{CampoDeVlemingGuerreiro:2026} on the K-stability of general weighted complete intersections provides further context for this question.

In the Calabi--Yau case our method gives discreteness uniformly, but we do not attempt to settle finiteness in complete generality. In dimension at least three, Grothendieck--Lefschetz arguments can often be combined with discreteness to obtain finiteness by controlling the Picard group. This approach is used for ordinary hypersurfaces (see \cite[Theorem 30]{Kollar:2019} for an exposition) and in the weighted-projective setting in Esser's work \cite{Esser:2024}. We defer a systematic treatment of finiteness for automorphism groups of Calabi--Yau hypersurfaces to future work.

Weighted projective spaces provide a natural first setting for these questions, which have been well studied over the complex numbers. For quasismooth well-formed weighted complete intersections over \(\mathbb C\), Przyjalkowski and Shramov \cite{PrzyjalkowskiShramov:2021} establish, under suitable dimension and canonical-class hypotheses, that the automorphism group is linear algebraic and that actions of its reductive subgroups extend to the ambient weighted projective space.
The first author \cite[Theorem 3.13]{Bunnett:2024} studied finiteness of the subgroup of Cartier hypersurface automorphisms induced by the ambient space.
The degree bound stated there required correction,  which is supplied by  \cite[Theorem 3.1]{Esser:2024}, which gives a precise criterion for finiteness of this subgroup in terms of the weights and degree and extended the result to any hypersurface, removing the Cartier condition.
In \textit{loc.\ cit.}, it is proven that every automorphism extends to the ambient space for quasismooth well-formed hypersurfaces of dimension at least three, and for surfaces with nontrivial canonical class \cite[Theorem 2.1]{Esser:2024}.
Together, these results explain both when ambient automorphisms capture the full group and when that group is finite.

Once finiteness is known, a further question is how large the automorphism group can be.
For smooth hypersurfaces in complex projective space in the range where the full automorphism group is finite,  \cite{YangYuZhu:2025} and  \cite{EsserLi:2025} determine the largest possible order for fixed dimension and degree.
Apart from explicitly described exceptions, the maximum is attained by the Fermat hypersurface.
In weighted projective space,  \cite{Esser:2024} also gives uniform upper bounds for the finite subgroup of automorphisms induced by the ambient space.
These results complement the question considered here, of when finiteness holds for every smooth or quasismooth member of a linear system.

Our method gives uniform discreteness under Esser's degree bounds in every tame characteristic, meaning here that the characteristic is coprime to the weights. In particular, we allow the characteristic to divide the hypersurface degree. This gives finiteness in the Fano case throughout this range of characteristics. We treat a well-formed quasismooth hypersurface as a smooth divisor in the associated smooth toric Deligne--Mumford stack, where the restriction and residue sequences are available, and use Bott-type vanishing to check the cohomological hypotheses.

\begin{theorem}[Weighted projective hypersurfaces]
Suppose that $X$ is a quasismooth degree $d$ well-formed hypersurface in a well-formed $\PP(a_0, \dots , a_n)$ with $n \geq 3$ and such that each $a_i$ is invertible in $k$.
We assume that $d \geq 2\cdot \max\{a_i\}$ with equality permitted only if the maximum weight occurs exactly once.

Then $\Aut_{X/k}$ is \'etale over \(k\) and finite if $X$ is Fano.
\end{theorem}

 The method also applies to more general toric varieties.
 In particular, Theorem \ref{thm:toric-fano-criterion} gives a criterion for hypersurfaces whose class is proportional to the anticanonical class of the ambient toric variety.

Hypersurfaces in products of projective spaces provide another important application, in which the hypersurface class is given by a multidegree. Here the Bott formula on each factor, together with the K\"unneth formula, allows us to translate the cohomological criterion into conditions on the multidegree and the dimensions of the factors. These conditions yield discreteness for every smooth hypersurface in the corresponding linear systems, and hence finiteness whenever the hypersurfaces are Fano.

\begin{theorem}[Hypersurfaces in products of projective spaces]
\label{thm:intro-products}
Let \(k\) be a field, and let
\[
X\subset Y=\prod_{j=1}^{r}\mathbb P_k^{n_j}
\]
be a smooth hypersurface of multidegree \((d_1,\ldots,d_r)\), where \(r\geq2\).
Suppose that either
\begin{enumerate}
\item \(n_j\geq2\) and \(d_j\geq2\) for every \(j\); or
\item \(r=2\), \(1\leq n_1<n_2\), \(d_1\geq1\), and \(d_2\geq2\), with
\[
(n_1,d_2)\neq(1,2)
\quad\text{or}\quad n_2\text{ even}.
\]
\end{enumerate}
Then \(H^0(X,T_X)=0.\)
In particular, the automorphism group scheme \(\Aut_{X/k}\) is discrete, and is finite if \(X\) is Fano.
\end{theorem}

We also look at cases where the required vanishings fail, although the automorphism group may still be finite. An interesting case is given by hypersurfaces which are projectivisations of syzygy bundles. Here the general smooth hypersurface has finite automorphism group, but some special smooth hypersurfaces have infinite automorphism group. We study these families and how their automorphism groups vary.

Finally, we apply these results to moduli of quasismooth hypersurfaces in a fixed toric variety. Uniform vanishing of global vector fields is useful here because it controls automorphisms throughout the quasismooth locus. Under additional reductivity and discriminant hypotheses, we show that the quotient by ambient automorphisms preserving the hypersurface class is a smooth separated Deligne--Mumford stack with a normal affine coarse moduli space. This gives a toric analogue of Benoist's construction for hypersurfaces in projective space \cite{Benoist:2013}.

The paper is organised as follows.
In Section 2 we recall the toric stack constructions and the vanishing results we need.
Section 3 develops the Kodaira--Spencer method and proves the main cohomological criterion.
In Section 4 we apply it to toric hypersurfaces whose class is proportional to the anticanonical class, and to weighted projective hypersurfaces. Section 5 treats hypersurfaces in products of projective spaces and examines some families where the criterion fails. Section 6 gives the application to moduli.

\section*{Acknowledgments}
 The authors thank MATH+ for supporting Caroline Namanya’s research visit to Technische Universität Berlin as a Hanna Neumann Fellow, funded by the Deutsche Forschungsgemeinschaft (DFG, German Research Foundation) under Germany's Excellence Strategy – The Berlin Mathematics Research Center MATH+ (EXC-2046/2, project ID: 390685689). 
 The second author would also like to thank the Institut für Mathematik at Technische Universität Berlin for its hospitality.
 
\section*{Notation and conventions}

Throughout, $k$ is an arbitrary field.
All toric varieties are normal and split: their tori are isomorphic over $k$ to a power of $\GG_m$.
We write $Y_{\sm}$ for the smooth locus of $Y$ over $k$ and
$Y_{\sing}=Y\setminus Y_{\sm}$.

For a smooth scheme or Deligne--Mumford (DM) stack $\cY$, we write
\[
\Omega^p_{\cY}=\bigwedge^p\Omega^1_{\cY/k},
\qquad
T_{\cY}
=\mathcal{H}om_{\cO_{\cY}}
  (\Omega^1_{\cY/k},\cO_{\cY}).
\]
When $\cY$ has pure dimension $n$, its canonical bundle is $\omega_{\cY}=\Omega^n_{\cY}$.
We set $\Omega^0_{\cY}=\cO_{\cY}$ and $\Omega^p_{\cY}=0$ for $p<0$; the same degree conventions apply to logarithmic forms.
For a coherent sheaf $\cF$, we write $h^q(Y,\cF)=\dim_k H^q(Y,\cF)$ whenever this dimension is finite.

For a coherent sheaf $\cF$ we set
\[\cF^{[p]} = \left(\bigwedge^p\cF\right)^{\vee\vee},\]
the reflexive hull of the $p^{th}$ exterior power.

For a projective scheme $X$, $\Aut_{X/k}$ denotes its automorphism group scheme; we call it discrete if it is \'etale over $k$.

\section{Toric orbifolds}

In this section we provide background on simplicial toric varieties and their canonical smooth toric stacks.
There are singular versions of the restriction and Poincaré residue short exact sequences; however, they either work in reduced generality, as in \cite{Steenbrink:1993}, or are more delicate and involve the mixed Hodge structure, as in \cite{BatyrevCox:1994}.
For this reason, we work with smooth toric DM stacks where these do hold.

\subsection{Background on toric geometry}

We recall the quotient construction of a toric variety, originally due to Cox \cite{Cox:1995}; we refer to \cite{ACampoHausenSchroer:2002} over arbitrary fields.
We fix a projective simplicial toric variety $Y$ over a field $k$ corresponding to a complete simplicial fan \(\Sigma\).

Let
\[S = k[x_0 , \dots , x_r] = \bigoplus_{\alpha \in \Cl(Y)}S_\alpha\]
be its Cox ring, where the variables are indexed by the rays $\Sigma(1) = \{\rho_0, \dots , \rho_r\}$.
The grading defines an action of the diagonalisable $k$-group scheme $D = \Hom(\Cl(Y),\GG_m)$ on $\AA^{r+1}$.

Moreover, there is a codimension 2 closed subset $Z  \subseteq \AA^{r+1}$, write $U = \AA^{r+1}-Z$, such that $\faktor{U}{D} \cong Y$ and is a geometric quotient.

We recall the construction of canonical toric orbifolds and their coarse moduli spaces from \cite[Section 4.2]{FantechiMannNironi:2010}.

\begin{definition}
Define $\cY := \left[\faktor{U}{D}\right]$ to be the smooth toric DM stack of $Y$ and $\pi :\cY \to Y$ its coarse moduli space.
We assume always that $\cY$ is tame; equivalently, the orders of its geometric stabilisers are invertible in \(k\).
\end{definition}

Then $\cY$ is a toric orbifold, that is, has generically trivial stabilisers, and is a canonical stack of $Y$ in the following sense as in \cite[Definition 4.4]{FantechiMannNironi:2010}.

\begin{definition}
    Let \(\cY\) be an irreducible \(n\)-dimensional smooth Deligne-Mumford stack.
    Let \(\pi : \cY \to Y\) be the coarse moduli space.
    The stack \(\cY\) will be called canonical if the locus where \(\pi\) is not an isomorphism has codimension at least \(2\).
\end{definition}

We now recall some basics on hypersurfaces in toric varieties.
Let \(Y\) be a projective simplicial toric variety, let \(\pi\colon \mathcal Y\longrightarrow Y\) be its canonical toric stack, and let \(X\subset Y\) be an effective Weil divisor.
The divisor \(X\) determines an effective Cartier divisor \(\cX \coloneqq \overline{\pi^{-1}(\bigl(X\cap Y_{\sm}\bigr))} \subset\cY\), whose coarse moduli space is \(X\).
We denote the induced coarse moduli map by \(\pi_X\colon\cX\longrightarrow X.\)

\begin{definition}
\label{definition:quasismooth-toric-hypersurface}
We say that \(X\subset Y\) is \emph{quasismooth} if the induced hypersurface \(\mathcal X\subset\mathcal Y\) is smooth.
\end{definition}

This is equivalent to the definition via the Jacobian criterion and via $V$-manifolds given in \cite{BatyrevCox:1994}.
Note that the definition does not imply that \(\mathcal X\to X\) is a canonical stack of \(X\): the induced stack may still have stack structure in codimension one.
Looking at this behaviour in weighted projective spaces leads to the notion of well-formed subvarieties \cite{Przyjalkowski:2024} which will be important later on. 

\begin{definition}[Well-formed hypersurface]
\label{definition:well-formed-hypersurface}
Let \(X\subset Y\) be an effective Weil divisor.
We say that \(X\) is \emph{well-formed in \(Y\)} if
\[\codim_X \bigl(X\cap Y_{\mathrm{sing}}\bigr)\geq 2.\]
\end{definition}

Equivalently, \(X\) is well-formed if it meets the stacky locus of the canonical stack \(\cY\) only in codimension at least \(2\) within \(X\).

\begin{lemma}
\label{lemma:well-formed-coarse-map}
Let \(X\subset Y\) be an irreducible, quasismooth Weil divisor, with induced hypersurface stack \(\cX\subset\cY\).
Then the following conditions are equivalent:
\begin{enumerate}
    \item \(X\) is well-formed in \(Y\);
    \item the coarse moduli map \(\pi_X\colon\mathcal X\to X\) is an isomorphism outside a closed subset of codimension at least \(2\) in \(X\), that is, $\cX$ is the canonical stack of $X$;
    \item \(\pi_X\) is an isomorphism in codimension one.
\end{enumerate}
\end{lemma}

\begin{proof}
By \cite[Remark~4.5]{FantechiMannNironi:2010}, the isomorphism locus of \(\pi\colon\mathcal Y\to Y\) is precisely \(\pi^{-1}(Y_{\sm})\).
Hence the image of the non-isomorphism locus of \(\pi_X\) is
    \[X\cap Y_{\sing}.\]
Thus \(\pi_X\) is an isomorphism in codimension one if and only if
\[\codim_X\bigl(X\cap Y_{\sing}\bigr)\geq 2.\]
Finally, quasismoothness means that \(\mathcal X\) smooth, so the last condition is precisely the definition of \(\mathcal X\) being canonical.
\end{proof}

\begin{remark}
\label{remark:weak-strong-well-formedness}
There are two related notions of well-formedness in the literature on weighted complete intersections; see \cite{Przyjalkowski:2024}.
Throughout this paper, \emph{well-formed} will be the stronger of the two notions:
    \[\codim_X(X\cap Y_{\mathrm{sing}})\geq2.\]
This is the condition we will need, since it is equivalent to the induced coarse moduli map \(\cX\to X\) being an isomorphism in codimension one.
Weak well-formedness does not ensure this.
\end{remark}

\subsection{Sheaves on toric stacks and varieties}

First we note that for canonical stacks, the pushforward of a reflexive sheaf is again reflexive.
The following is certainly well known, but we include a proof as we do not know a reference.

\begin{lemma}\label{lemma:pushforward-is-reflexive}
Suppose that $\cX$ is a smooth Noetherian DM stack over $k$ and $\pi: \cX \to X$ a scheme coarse moduli space which is an isomorphism outside a set of codimension at least 2.
For any reflexive $\cF \in \Coh(\cX)$ the pushforward $\pi_*\cF \in \Coh(X)$ is reflexive.
Moreover, $\pi_* : \Pic(\cX) \rightarrow \Cl(X)$ is an isomorphism.
\end{lemma}

\begin{proof}
Since $\cX$ is a smooth DM stack, $X$ is a normal irreducible scheme and so to prove reflexivity we must prove that $\pi_*\cF$ is torsion free and has the Harthogs extension property \cite[Tag 0AVB]{stacks}.
Torsion free is immediate, so we must only justify the extension property.

Consider an open $W \subseteq X$ and a closed subset $Z \subseteq W$ of codimension $\geq 2$.
Then \[\Gamma(W\setminus Z , \pi_*\cF) = \Gamma(\pi^{-1}(W) \setminus \pi^{-1}(Z) , \cF) = \Gamma(\pi^{-1}(W),\cF) = \Gamma(W,\pi_*\cF) \enspace ,\]
    since $\pi$ is a homeomorphism on geometric points.

    Applying this to a line bundle $\cL \in \Pic (\cX)$ says that $\pi_*\cL$ is a reflexive rank 1 sheaf.
    That this defines the claimed isomorphism, see \cite[Remark 4.5.2]{FantechiMannNironi:2010}.
\end{proof}

\begin{corollary}\label{cor:pushforward-p-forms}
    Consider a line bundle $\cL \in \Pic(\cX)$ and let $L = \pi_*\cL$, then
    \[\pi_*\left( \Omega_{\cX}^p \otimes \cL\right) \cong \Omega^{[p]}_{X}(L).\]
\end{corollary}

\begin{proof}
Since both sheaves are reflexive and $\pi$ is an isomorphism on an open with complement of codimension at least 2, they are  isomorphic.
\end{proof}

Lastly we have the following result which says that the construction  of sheaves on simplicial toric varieties can be expressed purely through this toric orbifold construction.

In \cite{Cox:1995}, it is showed that the category of quasicoherent sheaves on a simplicial toric variety may be constructed as the Serre quotient of the category of graded modules over the Cox ring, which is  termed the "homogeneous coordinate ring", just as Serre proved for projective spaces and  is made explicit by \cite{AurouxKatzarkovOrlov:2008} for weighted projective stacks and spaces.

We show that this construction can be expressed via the coarse moduli space.
The category of graded $S$-modules is equivalent to the category of $D$-equivariant sheaves on $\AA^{r+1}$.
Thus a graded $S$-module $M$ defines a quasicoherent sheaf on $\cY$. We denote it $\widetilde{M}_{\cY} \in \Qcoh(\cY)$.

Using Cox's construction, we also get a sheaf $\widetilde{M}_Y \in \Qcoh(Y)$.

\begin{theorem}\cite[Theorem 3.2, Proposition 3.3]{Cox:1995}
\label{theorem:cox-sheaves-on-toric-var}
    If $Y$ is a simplicial toric variety, then every quasicoherent sheaf on $Y$ is of the form $\widetilde{M}_Y$ for some graded $S$-module $M$.
    Moreover, if $M$ is finitely generated, then $\widetilde{M}_Y$ is coherent and every coherent sheaf arises in this way.
\end{theorem}

Moreover, Cox's theorem says that every sheaf has this form when $Y$ is simplicial.
For toric orbifolds, we shall show that these two sheaves are related by the pushforward along $\pi$.

\begin{proposition}\label{prop:Cox-sheaves-toric-orbifold}
    With the notation above, we have that $\pi_*\widetilde{M}_\cY = \widetilde{M}_Y$.
    Moreover, the functors $\pi_* : \Qcoh(\cX) \to \Qcoh(Y)$ and $\pi_*: \Coh(\cY) \to \Coh(Y)$ are essentially surjective.
\end{proposition}

\begin{proof}
Suppose that $\Sigma$ is the complete simplicial fan and pick a cone $\sigma \in \Sigma$.
Then the corresponding open set is denoted $Y_\sigma$ and the coarse moduli map restricted to this open set it
\[\pi : \left[ \faktor{U_{x^{\hat{\sigma}}}}{D} \right] \to  \faktor{U_{x^{\hat{\sigma}}}}{D} = Y_\sigma \]
is simply the geometric quotient and $x^{\hat{\sigma}}$ is a monomial, see \cite[Theorem 2.1]{Cox:1995} for details.

On this open set, the pushforward map is just taking invariants of the $D$-equivariant sheaf.
In the case of diagonalisable groups, this is just the $0^{\text{th} }$-graded piece of the corresponding module.
That is, given a graded $S$-module $M$, we have that
\[\pi_*\widetilde{M}_{\cY}(Y_\sigma) = \left( M_{x^{\hat{\sigma}}} \right)_0 = \widetilde{M}_Y(Y_\sigma) \enspace .\]
Thus the sheaves are actually equal.
Surjectivity is given by Theorem \ref{theorem:cox-sheaves-on-toric-var}.
\end{proof}

See \cite[Proposition 2.3]{AurouxKatzarkovOrlov:2008} for more details in the weighted projective stacks case.

\subsection{Bott vanishing on toric orbifolds}

Following \cite{KawakamiTotaro:2025}, we make the following definition and extend it to smooth DM stacks.

\begin{definition}

Let $X$ be a projective variety over a field. We say that $X$ satisfies \emph{Bott vanishing for ample Weil divisors} if for every ample Weil divisor $A$, we have that
\[H^q(X,\Omega^{[p]}_X(A))=0\]
for every $q>0, p\geq 0$.

For a smooth DM stack $\cX$ over a field $k$, we say that $\cX$ satisfies \emph{Bott vanishing} if for every ample divisor $\cA$, we have that
    \[H^q(\cX,\Omega^{p}_\cX(\cA))=0\]
    for every $q>0, p\geq 0$.
\end{definition}

We define $\cA$ on $\cX$ to be ample if $\pi_*\cA \in \Cl(X)$ is ample, as in \cite{AbramovichHassett:2011}.

Fujino proved the following  \cite[Proposition 3.2]{Fujino:2007}.

\begin{theorem}\label{theorem:Fujino-Bott-vanishing-toric-varieties}
    Given a projective toric variety $X$ over a field $k$.
    Then $X$ satisfies Bott vanishing for ample Weil divisors.
\end{theorem}

Then as a corollary to the above theorem we get the following.

\begin{corollary}\label{cor:bott-vanishing-toric-orbifold}
    Suppose that $\pi : \cX \to X$ is a toric orbifold, then $\cX$ satisfies Bott vanishing.
\end{corollary}

\begin{proof}
    Since $\pi$ is a coarse moduli map for a tame stack, it is cohomologically affine and hence we have that
    \[H^q(\cX, \Omega_\cX^{p}\otimes \cL) = H^q(X,\pi_*(\Omega_\cX^{[p]}\otimes \cL)) \cong H^q(X,\Omega_X^{[p]}(L)))\]
    by Corollary \ref{cor:pushforward-p-forms}.
    The result follows.
\end{proof}

\begin{remark}
    In \cite{KawakamiTotaro:2025}, the authors prove that a wider class of varieties satisfy Bott vanishing for ample Weil divisors.
    This means our method developed in Section \ref{section:KS-method} may be able to be applied to a wider class of varieties and their canonical stacks.

    Our proof of Corollary \ref{cor:bott-vanishing-toric-orbifold} passes to the coarse space and applies known vanishing there.
    It would be interesting to know if Fujino's proof works directly for smooth toric stacks and their stacky fans as defined in \cite{BorisovChenSmith:2005}.
\end{remark}

We record one other vanishing theorem we will need in the sequel.
For a proof, one can consult \cite[Theorem 9.3.2]{CoxLittleSchenck:2011}.

\begin{theorem}
\label{theorem:vanishing-with-no-twisting}
    Suppose that $X$ is a projective simplicial toric variety.
    Then if $p \neq q$, it holds
    \[H^q(X,\Omega_X^{[p]})=0 \enspace .\]
\end{theorem}

We dub this extra vanishing \emph{extended Bott vanishing}.

\begin{definition}
    We say that $X$ satisfies \emph{extended Bott vanishing} if it satisfies Bott vanishing for ample divisors and the vanishing given in Theorem \ref{theorem:vanishing-with-no-twisting}.

    The definition is analogous when $X$ is a stack.
\end{definition}

Extended Bott vanishing will be required for our method developed in Section \ref{section:KS-method}.

\subsection{Logarithmic Differentials on smooth DM stacks}\label{section:stacky-ses}

Here we state and prove the short exact sequences we need.
Both are doubtless known to experts and the proofs are applications of well-known results in the setting of logarithmic geometry \cite{EsnaultViehweg:1992} and \'etale descent.
However, since we have no exact citation, and to make the exposition as accessible as possible, we include proofs here.

Let $\cY$ be a smooth, separated DM stack over \(k\) and $\cX \subset \cY$ be an effective Cartier divisor, smooth over \(k\).
Suppose that \(U \to \cY\) is an \'etale atlas for \(\cY\) and denote by $D \coloneqq \cX \times_\cY U$ the Cartier divisor lying above $\cX$.
We write \(R=U\times_\cY U \rightrightarrows U\) for the groupoid presentation of $\cY$ with \'etale maps $s,t : R \to U$ and since $\cY$ is a separated DM stack, $R$ is a scheme.


\begin{remark}
Separatedness is used in this section only to ensure that the overlaps are schemes; it can be omitted by carrying out the same construction on \'etale scheme charts of \(R\).
We make the restriction to shorten arguments and since we are only interested in the sequel in the separated case.
\end{remark}

Then \(D_R \coloneqq s^*D = t^*D \subset R\) is also a smooth effective Cartier divisor in $R$.
Equip the pairs \((R,D_R)\) and \((U,D)\) with their divisorial log structures.
Since $D_R=s^*D=t^*D$, both maps are strict;
since their underlying maps are \'etale, they are log \'etale.

Then we consider $\Omega^1_U(\log D)$, the sheaf of differentials with logarithmic poles along $D$ (all over \(k\)) \cite[0FMU]{stacks}.

This then forms a sheaf on $\cY$.
Indeed, since $D$ is constructed from a divisor on $\cY$, we know that both groupoid maps can be considered as maps of log spaces $s,t:(R,D_R) \to (U,D)$.
Their relative logarithmic cotangent sheaves therefore vanish, and the short exact sequence of
\cite[Proposition 3.14(2)]{AbramovicEtAl:2013},
applied over $k$ with its trivial log structure, gives
the canonical isomorphisms
\[\beta_s : s^*\Omega^1_{U}(\log D) \xrightarrow{\cong}\Omega^1_{R}(\log D_R), \quad \beta_t : t^*\Omega^1_U(\log D) \xrightarrow{\cong}\Omega^1_R(\log D_R).\]
Thus we have
\[\varphi := \beta_t^{-1}\beta_s : s^*\Omega^1_{U}(\log D) \xrightarrow{\cong}t^*\Omega^1_{U}(\log D).\]
Moreover, the cocycle condition follows from the compatibility
$\beta_{f\circ g}=\beta_g\circ g^*\beta_f$
applied to the composites
\[
U\times_{\cY}U\times_{\cY}U
\xrightarrow{p_{ij}} R
\xrightarrow{s,t} U.
\]
Indeed, writing $q_i$ for the projection onto the $i$th factor,
we have $s\circ p_{ij}=q_i$ and $t\circ p_{ij}=q_j$, so that
\[
p_{ij}^*\varphi=\beta_{q_j}^{-1}\beta_{q_i}.
\]
Consequently,
\[
p_{23}^*\varphi\circ p_{12}^*\varphi
=\beta_{q_3}^{-1}\beta_{q_2}
 \beta_{q_2}^{-1}\beta_{q_1}
=\beta_{q_3}^{-1}\beta_{q_1}
=p_{13}^*\varphi.
\]

Hence by descent \cite[06WU]{stacks}, $\Omega^1_{U}(\log D)$ descends to a locally free sheaf on $\cY$ denoted by $\Omega^1_{\cY}(\log \cX)$.

We now prove that the Poincaré residue and restriction short exact sequences descend to the stack.
Denote the inclusion $i : \cX \hookrightarrow \cY$ and we let
\[\Omega_\cY^p(\log \cX) := \bigwedge^p  \Omega_\cY^1(\log \cX) \quad p \geq 0,\]
and declare the sheaves are $0$ if $p<0$.

\begin{lemma}\label{lemma:poincare-residue-ses}
    The following sequence is exact for every \(p\geq 0\).
    \begin{equation}
    0 \longrightarrow \Omega^{p}_\cY \longrightarrow \Omega^{p}_\cY(\log \cX) \longrightarrow i_{\ast}\Omega^{p-1}_\cX  \longrightarrow 0, 
\end{equation}
where $\Omega^{-1}_{\cX/k}=0$.
\end{lemma}

\begin{proof}
Write $a:U\to\cY$ for the chosen \'etale atlas and $j:D\hookrightarrow U$ for the pullback of $i$.
The residue sequence for the smooth scheme pair $(U,D)$ is
\[
0 \longrightarrow \Omega^p_{U/k}
\longrightarrow \Omega^p_{U/k}(\log D)
\xrightarrow{\operatorname{Res}_D}
j_*\Omega^{p-1}_{D/k}
\longrightarrow 0;
\]
see \cite[Tag 0FMW]{stacks}.

We verify that its maps commute with strict \'etale pullback.
Compatibility of the inclusion of regular forms is immediate.
For residue, let $f:(V,E)\to(U,D)$ be an \'etale morphism with $E=f^*D$, and let $f_E:E\to D$ be the induced map.
Locally, if $t$ defines $D$, then
\[
\operatorname{Res}_D
\left(\alpha+\frac{dt}{t}\wedge\eta\right)
=\eta|_D
\]
for regular forms $\alpha$ and $\eta$.
Since $f^*t$ defines $E$, this formula gives
\[
\operatorname{Res}_E
\left(f^*\alpha+
\frac{d(f^*t)}{f^*t}\wedge f^*\eta\right)
=f_E^*(\eta|_D).
\]
Thus residue commutes with pullback under the canonical
identifications of logarithmic forms and of forms on
the divisors.

Applying this compatibility to $s,t:R\rightrightarrows U$,
the maps in the scheme sequence respect the descent data.
Moreover, base change for the closed immersion $i$ and
\'etale pullback of ordinary differentials give
\[
a^*i_*\Omega^{p-1}_{\cX/k}
\cong j_*\Omega^{p-1}_{D/k}.
\]
The sequence therefore descends to the displayed sequence
on $\cY$ by \cite[Tag 06WU]{stacks}.
Its exactness can be checked after pullback along the
\'etale surjective atlas $a$, where it is the scheme
residue sequence.
\end{proof}
\begin{lemma}\label{lemma:restriction-ses}
    The following sequence is exact for every \(p\geq0\).
    \begin{equation}
    0 \longrightarrow \cO_\cY(-\cX) \otimes \Omega^{p}_\cY(\log \cX) \longrightarrow \Omega^{p}_\cY \longrightarrow i_{\ast}\Omega^{p}_\cX \longrightarrow 0,
\end{equation}
\end{lemma}

\begin{proof}
For smooth scheme pairs, this is the restriction sequence of \cite[Properties 2.3(c)]{EsnaultViehweg:1992}.
The maps are multiplication
\[
\cO_U(-D)\otimes\Omega^p_{U/k}(\log D)
\longrightarrow\Omega^p_{U/k}
\]
and restriction of differential forms to $D$.
Both commute with strict \'etale pullback.
Using also $a^*\cO_{\cY}(-\cX)\cong\cO_U(-D)$, the sequence therefore descends and is exact by the same argument as in Lemma~\ref{lemma:poincare-residue-ses}.
\end{proof}

\section{The Kodaira--Spencer method}\label{section:KS-method}

In \cite{KatzSarnak:1999}, Katz and Sarnak prove finiteness of automorphism groups of smooth hypersurfaces in projective space over any field using the method of Kodaira and Spencer \cite{KodairaSpencer:1958} combined with vanishing results of Deligne.
We generalise this method as far as possible.
The structure of the argument in this section is taken straight from \cite{KodairaSpencer:1958} and \cite{KatzSarnak:1999}. What is new here is that, with a bit of bookkeeping, these powerful general principles apply in a much wider context.

Let $\cY$ be a smooth DM stack of dimension $n$ over a field $k$ and $i : \cX \subset \cY$ a smooth Cartier divisor with coarse moduli space $X$.

We introduce the following family of line bundles for each $q \geq 0$
\[\cL_q \coloneqq \omega_\cY^{\otimes -1} \otimes \cO_\cY(-(q+1)\cX) \in \Pic(\cY) \enspace .\]
Note that $\cL_q \otimes \cO_\cY(-\cX) \cong \cL_{q+1}$.
We give conditions under which $H^0(\cX, T_\cX) = 0$.

Since $\cX$ and $\cY$ are smooth,
\begin{equation}\label{equ:AlternatingAlgebra}
    T_{\cX} \cong \Omega^{n-2}_{\cX} \otimes \omega_\cX^{\otimes -1},
\end{equation}
and the same holds for $\cY$.

Adjunction gives $\omega_\cX \cong (\omega_\cY \otimes_{\cO_\cY} \cO_\cY(\cX))|_\cX$, and using that the conormal bundle of a Cartier divisor satisfies $\mathcal{N}^*_{\cX/\cY} \cong \cO_\cY(-\cX)|_\cX$, we obtain $\omega_\cX^{\otimes -1} \cong (\omega_\cY^{\otimes -1} \otimes \cO_\cY(-\cX))|_\cX = i^*\cL_0$.
Thus, via \eqref{equ:AlternatingAlgebra}, we want to prove that $H^0(\cX, \Omega^{n-2}_\cX \otimes \cL_0|_\cX) = 0$, which is the case $q=0$ of the statement $\mathbf{C}(q)$ below.

This is exactly the statement which Kodaira--Spencer and Katz--Sarnak proved via descending induction.
We will use it together with three further vanishing statements:
\begin{align}
    \tag{$\mathbf{A}(q)$}\label{Aq} H^q(\cY, \Omega_\cY^{n-q-1}(\log \cX) \otimes \cL_q) &= 0, \\
    \tag{$\mathbf{B}(q)$}\label{Bq} H^q(\cY, \Omega_\cY^{n-q-2} \otimes \cL_q) &= 0, \\
    \tag{$\mathbf{C}(q)$}\label{Cq} H^q(\cX, \Omega^{n-q-2}_\cX \otimes \cL_q|_\cX) &= 0, \\
    \tag{$\mathbf{D}(q)$}\label{Dq} H^q(\cY, \Omega^{n-q-1}_\cY \otimes \cL_q) &= 0.
\end{align}

To relate these vanishing statements and piece together Kodaira and Spencer's descending induction we use the two short exact sequences from Section \ref{section:stacky-ses}.

\begin{lemma}\label{lemma:B+A=C}
    $\mathbf{B}(q)$ and $\mathbf{A}(q+1)$ imply $\mathbf{C}(q)$.
\end{lemma}

\begin{proof}
    We twist the restriction short exact sequence \ref{lemma:restriction-ses} by $\cL_q$ and consider the long exact sequence in cohomology
    \[\cdots \to H^{q}(\cY, \Omega^{n-q-2}_\cY \otimes \cL_q) \to H^q(\cX, \Omega_\cX^{n-q-2} \otimes i^{\ast}\cL_q) \to H^{q+1}(\cY, \Omega_\cY^{n-q-2}(\log \cX) \otimes \cL_{q+1}) \to \cdots\]
    from which the result follows.
\end{proof}

\begin{lemma}
    \ref{Cq} and \ref{Dq} imply \ref{Aq}.
\end{lemma}

\begin{proof}
    Consider the long exact sequence coming from the Poincar\'e residue short exact sequence \ref{lemma:poincare-residue-ses} twisted by $\cL_q$:
    \[\cdots \to H^q(\cY,\Omega^{n-q-1}_\cY \otimes \cL_q) \to H^q(\cY, \Omega^{n-q-1}_\cY(\log \cX) \otimes \cL_q) \to H^q(\cX, \Omega^{n-q-2}_\cX \otimes i^{\ast}\cL_q) \to \cdots \enspace .\]
    The result follows.
\end{proof}

Combining these two lemmas, we get the `Kodaira--Spencer' method.

\begin{corollary}\label{cor:KS-method}
    If \ref{Bq} and \ref{Dq} hold for every $q$, then \ref{Cq} holds for every $q$.
\end{corollary}

\begin{proof}
    Given \ref{Bq} and \ref{Dq}, the two preceding lemmas give $\mathbf{A}(q+1) \implies $\ref{Cq} $\implies$ \ref{Aq}.
    Since for $q>n$ we know \ref{Aq} holds, we conclude that \ref{Aq} and \ref{Cq} hold for all $q$.
\end{proof}

We can then collect this all into the following theorem.

\begin{theorem}
\label{thm:KS-method}
Let \(k\) be a field and let \(\cY\) be a smooth, separated Deligne--Mumford stack over \(k\) of dimension \(n\), where \(n\geq 2\). 
Let \(i:\cX\hookrightarrow\cY\) be a smooth, effective Cartier divisor.

Let \(\pi_X:\cX\to X\) be the coarse moduli space.
Suppose that \(\pi_X\) is an isomorphism outside a closed subset of codimension at least \(2\).
Assume that \ref{Bq} and \ref{Dq} hold for every $q$.

Then
\[H^0(\cX,T_{\cX})\cong H^0(X,T_X)=0.\]
Moreover, if $X$ is Fano or general type, then $\Aut(X)$ is finite.
\end{theorem}

\begin{proof}
    Corollary \ref{cor:KS-method} gives us that $H^0(\cX,T_\cX) = 0$.
    Then we apply Lemma \ref{lemma:pushforward-is-reflexive} as $\pi_*T_\cX \cong T_X$ is reflexive to get
    \[H^0(\cX,T_\cX) =H^0(X,T_X)= 0.\]

    Lastly, we remark that if $K_X$ is anti-ample or big and $\QQ$-Cartier, then $\Aut_{X/k}$ is of finite type.
    See for example \cite[Lemma 2.3]{Brion:2022}.
    Since $H^0(X,T_X)=0$, it is discrete and thus finite.
\end{proof}

One very nice feature of the method is that it proves discreteness / finiteness for every smooth element in the linear system.
In fact, it proves more, since $H^0(X,T_X)=0$ holds for every well-formed, quasismooth element in the linear system, we get the immediate corollary as in \cite[Corollary 11.8.4]{KatzSarnak:1999}.

\begin{corollary}
    In the situation of the above theorem, 
    let $U \subset |X|$ be the locus of quasismooth elements of this linear system and $\mathcal{U}\to U$ be the universal quasismooth hypersurface.
    
    The $U$-group scheme $\Aut_U(\mathcal U) \to U$ is unramified.
\end{corollary}

\section{Toric Hypersurfaces}

Thus we have now reduced the question of discreteness to showing that \ref{Bq} and \ref{Dq} hold.

\subsection{Hypersurfaces which are proportional to $K_Y$}\label{subsection:toric-hyps}

Assume now that $Y$ is a projective simplicial toric variety of dimension $n$ and let $\cY$ be its smooth canonical toric orbifold.
When an ample hypersurface $X \subset Y$ is proportional to $K_Y$, we give a method of checking if $\Aut(X)$ is discrete, so long as the proportionality is outside a bad window.

Suppose that $X$ is a well-formed quasismooth ample hypersurface such that $-K_Y \sim_\QQ rX$ for some rational $r\in \QQ_{>0}$.
This will force the $\cL_q \in \Pic(\cY)$ to be either ample, anti-ample or torsion, which enables us to apply vanishing theorems if we avoid certain bad cohomological windows.

\begin{theorem}[Discreteness for $K_{Y}$-proportional hypersurfaces]\label{thm:toric-fano-criterion}
In the above situation, assume that $\mathbf{B}(0)$ and $\mathbf{D}(0)$ hold for $\cY$.
If \(r\neq\left\lceil\frac n2\right\rceil\), then
\[H^0(X,T_X)=0,\]
and if $K_X$ is ample or anti-ample then $\Aut_{X/k}$ is a finite group scheme.
\end{theorem}

If \(0<r<1\), then \(K_X\) is ample and the conditions \(\mathbf{B}(0)\) and \(\mathbf{D}(0)\) hold automatically.
In this case, finiteness is already covered by Matsumura's theorem \cite{Matsumura:1963}, at least in characteristic 0.
The additional conditions are needed in the Fano range \(r>1\).

\begin{proof}
By hypothesis $[\cL_q]=(r-(q+1))[\cX]$ in $\Pic(\cY)_{\QQ}$.
For $q\geq1$: if $q+1<r$ then $\cL_q$ is ample and if $q+1>r$ then $\cL_q$ is anti-ample; either way $\mathbf{B}(q)$ and $\mathbf{D}(q)$ hold, by Corollary~\ref{cor:bott-vanishing-toric-orbifold} or Serre Duality applied to it, since 
\[H^{q}(\cY,\Omega_\cY^{p}\otimes \cL_q)^{\vee} \cong H^{n-q}(\cY,\Omega_\cY^{n-p}\otimes \cL_q^{-1}).\]

It remains to treat $q+1=r$, where $\cL_q^m\cong\cO_{\cY}$ for some power $m$.
Since $\cY$ is a smooth proper toric DM stack, $H^q(\cY,\Omega_{\cY}^p)=0$ for $p\neq q$ by Theorem \ref{theorem:vanishing-with-no-twisting}, so $\mathbf{B}(q)$ can fail only if $n-q-2=q$ and $\mathbf{D}(q)$ only if $n-q-1=q$.
However, in either case this would force $r=q+1=\lceil n/2\rceil$.

Hence the hypotheses of Theorem~\ref{thm:KS-method} hold, giving $H^0(\cX,T_{\cX})\cong H^0(X,T_X)=0$.
\end{proof}

\subsection{Weighted Projective Spaces}

In the Picard rank one case, every hypersurface is proportional to the canonical bundle.
In this section, we settle when the hypothesis of Theorem \ref{thm:toric-fano-criterion} holds for weighted projective spaces.

The finiteness results for the automorphism groups were first proven by Esser for every quasismooth hypersurface over $\CC$ in \cite{Esser:2024}. Our theorem proves that the finiteness holds in every tame characteristic.

\noindent\textbf{Bott formula for weighted projective spaces}

The Bott formula was computed in \cite[2.3.2]{Dolgachev:1982}.
We assume, as in \cite{Dolgachev:1982}, that the characteristic of the field divides none of the weights.
This is exactly asking that the canonical toric stack is tame.

We first fix notation.
Let $\PP = \PP(a_0, \dots , a_n)$ be a well-formed weighted projective space.
For every $I \subset \{0,\dots , n\}$ define $|\underline{a}_I| = \sum_{i\in I}a_i$ and for a sheaf $\cF$ denote $h^q(\cF) = \dim_kH^q(\PP , \cF)$.

\begin{lemma}\cite[2.3.2]{Dolgachev:1982}
\label{lemma:wps-bott}
	In the notation above we have the following.
	\begin{itemize}
		\item For $p>0$ and $\ell \neq 0$, we have
		\[h^0(\Omega_{\PP}^{[p]}(\ell)) = \left( \sum_{|I| = p}h^0(\cO_\PP(\ell-|\underline{a}_I|)) \right) - h^0 \left(\Omega_{\PP}^{[p-1]}(\ell)\right) .\]
        \item For $p>0$ we have $h^0(\Omega^{[p]}) =0$,
		\item For $q \neq 0,p,n$ and any $\ell \in \ZZ$, we have $h^q(\Omega_\PP^{[p]}(\ell)) = 0$.
		\item For $p=0, \dots , n$, we have $h^p(\Omega_\PP^{[p]})=1$.
		\item For $p=1,\dots , n-1$ and $\ell \neq 0$ we have $h^p(\Omega_{\PP}^{[p]}(\ell)) = 0$.
		\item For $p=n$, we have \[h^n(\Omega_\PP^{[p]}(\ell)) =  \left( \sum_{|I| = n+1-p} h^0(\cO_\PP(-\ell-|\underline{a}_I|)) \right) - h^n(\Omega_\PP^{[p-1]}(\ell)) \enspace .\]
	\end{itemize}
\end{lemma}



We assume the following bounds for the next two results due to \cite{Esser:2024} in the complex case.
\begin{assumption}
    We assume that $d \geq 2\cdot \max\{a_i\}$ with equality permitted only if the maximum weight occurs exactly once.
\end{assumption}

We need the following lemma to rule out one problematic case where extended Bott vanishing does fail where we need it.
However, in these cases there exists no quasismooth hypersurface in the linear system.

\begin{lemma}\label{lemma:annoying}
	Consider a well-formed weighted projective space $\PP=\PP(a_0, \dots , a_n)$ with $n\geq 4$ and even.
	Then there exists no quasismooth hypersurface of degree $d=\frac{2}{n}\sum a_i$.
\end{lemma}

\begin{proof}
Let the weights be $a_0 \le a_1 \le \dots \le a_n$. We show the Iano-Fletcher criterion for quasismoothness \cite[Theorem 8.1]{IanoFletcher:2000} fails at the single-variable stratum $I=\{n-1\}$. It suffices to prove that $a_{n-1} \nmid d$ and $a_{n-1} \nmid (d-a_j)$ for all $j \neq n-1$.

Let $\phi = d - 2a_{n-1}$. We first claim $0 < \phi < a_0$. 
The degree assumptions imply $d>2a_{n-1}$, hence $\phi>0$.
Bounding the sum of $n-1$ terms gives $\sum_{i=1}^{n-1} a_i \ge \frac{n-1}{2}d - a_0$, yielding $(n-1)a_{n-1} \ge \frac{n-1}{2}d - a_0$. Rearranging provides $(n-1)\phi \le 2a_0$. Since $n \ge 4$, this forces $\phi \le \frac{2}{3}a_0 < a_0 \le a_{n-1}$. 
Writing $d = 2a_{n-1} + \phi$ with $0 < \phi < a_{n-1}$ immediately shows $a_{n-1} \nmid d$.

To show $a_{n-1} \nmid (d-a_j)$, we strictly trap $d-a_j \in (a_{n-1}, 2a_{n-1})$ for all $j \neq n-1$:
\begin{itemize}
    \item Upper bound (all $j$): $a_j \ge a_0 > \phi = d - 2a_{n-1} \implies d-a_j < 2a_{n-1}$.
    \item Lower bound ($j < n-1$): $a_j \le a_{n-1} \implies d-a_j \ge d-a_{n-1} > a_{n-1}$ (since $\phi > 0 \implies d > 2a_{n-1}$).
    \item Lower bound ($j = n$): By hypothesis, $d \ge 2a_n$, with strict inequality if $a_{n-1}=a_n$. In either case, $d-a_n > a_{n-1}$.
\end{itemize}

Since $d-a_j$ lies strictly between consecutive multiples of $a_{n-1}$, it cannot be a multiple. Thus, the criterion fails at $I=\{n-1\}$.
\end{proof}

Now one checks the Bott formula for hypersurfaces with degrees respecting Esser's bounds and we generalise his result to a field of characteristic coprime to the weights - one should note that his proof actually works for any characteristic where the weights and the degree are invertible, so our generalisation is just to allow cases where the characteristic divides the degree:

\begin{theorem}\label{theorem:main-wps}
Suppose that $X$ is a quasismooth degree $d$ well-formed hypersurface in a well-formed $\PP(a_0, \dots , a_n)$ with $n \geq 3$ and such that each $a_i$ is invertible in $k$.
We assume that $d \geq 2\cdot \max\{a_i\}$ with equality permitted only if the maximum weight occurs exactly once.

Then $\mathbf{B}(q)$ and $\mathbf{D}(q)$ are satisfied for all $q$.
In particular, the automorphism group schemes are \'etale over \(k\).
\end{theorem}

\begin{proof}
According to Lemma \ref{lemma:wps-bott}, there are only a few cases which we have to check.
First, we note that the $q=n$ case trivially hold since they are the cohomology of sheaves $\Omega_\PP^{[p]}$ with $p<0$.
	
Secondly, let us assume that there exists a $0\leq q < n$ such that we are in the `problem window' of $\sum_i a_i = (q+1)d$.
So that some multiple of $X$ is anticanonical.
By Lemma \ref{lemma:wps-bott}, we can only have non-vanishing in the cases $q=p$, that is, $H^p(\PP, \Omega_\PP^{[p]})$ which additionally forces $n-1-q = q$ for $\mathbf{D}(q)$ or $n-2-q=q$ for \(\mathbf{B}(q)\).	
	
Now let's look at each of these cases.
For $\mathbf{D}(q)$ we have $q = \frac{n-1}{2}$ and thus $\sum_i a_i = \frac{n+1}{2}d$.
Thus $d = \frac{2}{n+1}\sum_i a_i \leq 2 \max{a_i}$ with equality if and only if all $a_i = 1$. This is impossible by our assumptions on $d$.
	
For $\mathbf{B}(q)$, $n$ must be even and hence at least 4 and we have similarly that $d = \frac{2}{n}\sum_i a_i$ but in this case we are exactly in the case of Lemma \ref{lemma:annoying}, thus there exists no quasismooth element in such a linear system.
\end{proof}

\begin{remark}
\begin{enumerate}
\item We note that the case where failure of $\mathbf{B}(\frac{n-2}{2})$ could appear within the degree bounds, for example $(\underline{a},d) = ((3,4,4,4,5),10)$.
So these would indeed pose a problem for our method, but luckily the lemma ensures us that exactly in these cases, there can be no quasismooth element.

\item We also note that the hypothesis $n \geq 3$ is necessary.
For a quasismooth hypersurface $X$ in a well-formed weighted projective \emph{plane} $\PP(a_0,a_1,a_2)$, the relevant vanishing is \(\mathbf{B}(0)\) :
	\[
 H^0\!\left(\PP,\, \Omega^{[0]}_\PP \otimes \cL_0\right) = H^0(\PP,\, \cO_\PP(\textstyle\sum a_i - d)),
	\]
which is nonzero whenever $d \leq \sum a_i$.
The degree bound $d \geq 2\max\{a_i\}$ does not prevent this: for instance $\PP(1,1,2)$ with $d = 4$ satisfies $d = 2\cdot 2 = 2\max\{a_i\}$ and $d = \sum a_i = 4$ simultaneously, so $X$ is an anticanonical (genus-1) curve with infinite automorphism group.
More generally, a smooth curve in a weighted projective plane has finite automorphism group if and only if it has genus $\geq 2$, which for a quasismooth hypersurface of degree $d$ in $\PP(a_0,a_1,a_2)$ is equivalent to $d > \sum_i a_i = -K_\PP$.
There is nothing particularly interesting here: quasismooth curves are smooth and fall into the trichotomy of rational, elliptic or general type, so this is just ensuring that our curve is general type.
\end{enumerate}
\end{remark}


\section{Hypersurfaces in products of projective spaces}

In this section we apply the method described above to products of projective spaces.
We prove that the situation for hypersurfaces in $\PP^n$ generalises under certain conditions on the multi-degrees.
The situation is more delicate when one allows some of the degrees to be equal to 1.

\subsection{Bott formulas for products}

First we recall some standard results for computing Bott formulas for product varieties.

\begin{lemma}
    Given a product of smooth varieties $Y = \prod_{i=1}^{r}X_i$ and a line bundle $\cL = \boxtimes_{i=1}^{r} \cL_i \in \Pic(Y)$, we have
    \[H^q(Y,\Omega_Y^p(\cL)) \cong \bigoplus_{\substack{q_1+\cdots+q_r=q\\ p_1+\cdots+p_r=p}} \bigotimes_{i = 1}^r H^{q_i}(X_i , \Omega_{X_i}^{p_i}(\cL_i)) \enspace .\]
\end{lemma}

\begin{proof}
    Label the projection onto each factor by $\pr_i : Y \to X_i$ and note that cotangent sheaf splits $\Omega_Y^1 \cong \bigoplus_i \pr_i^*\Omega_{X_i}^1$.
    Then passing to $p$-forms and tensoring by $\cL$ gives
    \[\Omega_Y^p(\cL) \cong \bigoplus_{p_1 + \cdots + p_r=p}\bigotimes_{i=1}^r\pr_i^*\Omega^{p_i}_{X_i}(\cL_i)\]
    by standard properties of the alternating product.
    Hence
    \[H^q(Y,\Omega_Y^p(\cL)) = \bigoplus_{p_1+\cdots + p_r = p}H^q(Y,\bigotimes_{i=1}^r\pr_i^*\Omega_{X_i}^{p_i}(\cL_i))\]
    Then we apply the K\"unneth formula \cite{Kempf:1980} to get the desired result
    \[H^q(Y,\Omega_Y^p(\cL)) \cong \bigoplus_{\substack{q_1+\cdots +q_r = q \\p_1+\cdots + p_r = p}}\bigotimes_{i=1}^rH^{q_i}(X_i,\Omega_{X_i}^{p_i}(\cL_i))\]
\end{proof}

Let us introduce some notation to make the calculations easier to handle.
Define $b_{q,p,\cL}^Y = \dim H^q(Y,\Omega_Y^p(\cL))$ and in the same way
$b_{q_i,p_i,\cL_i}^{X_i} = \dim H^{q_i}(X_i , \Omega_{X_i}^{p_i}(\cL_i))$.

Then the previous lemma says that
\[b_{q,p,\cL}^Y = \sum_{\substack{q_1+\cdots+q_r=q\\ p_1+\cdots+p_r=p}}~\prod_{i=1}^rb_{q_i,p_i,\cL_i}^{X_i} \enspace .\]

Since each $b$ is a non-negative number, we must show that in each of the products, there is at least one $b$ factor which is 0.

\subsection{The main technical lemma}

Let $Y = \PP^{n_1} \times \cdots \times \PP^{n_r}$ with $r>1$ and consider a hypersurface $X$ of degree $(d_1, \dots , d_r)\in \ZZ^r_{>0}$.
Then for each $q$ we have each $\cL_q = \cO_{\PP^{n_1}}(\ell_1)\boxtimes \cdots \boxtimes \cO_{\PP^{n_r}}(\ell_r)$ where $\ell_i = (n_i+1)-(q+1)d_i$.

To make the notation more manageable, let $b^{\PP^{n_i}}_{q_i,p_i,\cO_{\PP^{n_i}}(\ell_i)} = b^i_{q_i,p_i}$, we may suppress the data of the line bundle since it is always the same.

Let us first repeat standard Bott vanishing for projective space here.
One can read it off from Lemma \ref{lemma:wps-bott}, however the emphasis is slightly different.

\begin{theorem}[Theorem 1.1 in \cite{SGA7II}]\label{theorem:Bott-vanishing-Pn}
Let $k$ be any field.
Then $H^q(\PP^n,\Omega^p_{\PP^n}(\ell))=0$ except in any of the following cases:
    \begin{align*}
 & q=0,\text{ and either }\ell>p\geq 0 \text{ or } \ell=p=0, \tag{$\Low$} \\
 & 1\leq q \leq n-1, \,\, p=q \text{ and } \ell=0, \tag{$\Mid$}\\
 & q = n , p\geq0 \text{ and either }p-\ell \geq n+1 \text{ or both } p=n \text{ and } \ell=0. \tag{$\Up$}
    \end{align*}
\end{theorem}

The following is our main technical lemma used in the proofs.

\begin{lemma}\label{lemma:technical-lemma}
We maintain the notation as above.
A monomial $b^{1}_{q_1,p_1} \cdots \, b^{r}_{q_r,p_r} \neq 0$ defines a partition
\[\{1,\dots ,r\} = \Low \sqcup \Mid\sqcup \Up,\]
where
\begin{align*}
 \Low^{0} &= \left\{ j \in \{1,\dots,r\} \;\middle|\; q_j = 0,\ p_j = \ell_j = 0 \right\}\\
\Low'    &= \left\{ j \in \{1,\dots,r\} \;\middle|\; q_j = 0,\ \ell_j \ge p_j + 1 \right\}\\
\Up^{0} &= \left\{ j \in \{1,\dots,r\} \;\middle|\; q_j = n_j,\ p_j = n_j,\ \ell_j = 0 \right\}\\
 \Up'  &= \left\{ j \in \{1,\dots,r\} \;\middle|\; q_j = n_j,\ p_j - \ell_j \ge n_j + 1 \right\}\\
        \Mid     &= \left\{ j \in \{1,\dots,r\} \;\middle|\; \ell_j = 0,\ 1 \le q_j \le n_j - 1,\ q_j = p_j \right\}
    \end{align*}
    Set $\Low = \Low^{0} \sqcup \Low'$ and $\Up = \Up^{0} \sqcup \Up'$.
\end{lemma}

\begin{proof}
 If such a monomial is non-zero, then every $b_{q_j,p_j} \neq 0$ for $j=1, \dots , r$ and so applying Theorem \ref{theorem:Bott-vanishing-Pn} gives the result.
\end{proof}

\subsection{Applying the method}

Let $Y = \PP^{n_1} \times \PP^{n_2}$ such that $1 \leq n_1<n_2$.
Then applying Lemma \ref{lemma:technical-lemma} shows that a non-vanishing Bott summand breaks into nine cases as shown in the table below.
\begin{table}[h]
\centering

\begin{tabular}{ |c |c |c |}
\hline
 \((\Low_{n_{1}},\Low_{n_{2}})\) & \((\Low_{n_{1}},\Mid_{n_{2}})\)& \((\Low_{n_{1}},\Up_{n_{2}})\) \\ 
 \hline
 \((\Mid_{n_{1}},\Low_{n_{2}})\) & \((\Mid_{n_{1}},\Mid_{n_{2}})\) & \((\Mid_{n_{1}},\Up_{n_{2}})\) \\ 
 \hline
 \((\Up_{n_{1}},\Low_{n_{2}})\) & \((\Up_{n_{1}},\Mid_{n_{2}})\) & \((\Up_{n_{1}},\Up_{n_{2}})\) \\
 \hline
\end{tabular}
\caption{Non-vanishing options for a Bott summand}
\label{table:PnxPm}
\end{table}

An entry in the table represents the partition of $\{1,2\}$ into the sets $\Low,\Mid$ and $\Up$ from Lemma \ref{lemma:technical-lemma}.

So that by $(\Mid_{n_1},\Low_{n_2})$ is meant that $1 \in \Mid$ and $2 \in \Low$ describing the type of non-vanishing for factor in the Bott summand $b_{q_1,p_1}^{\PP^{n_1}} \cdot b_{q_2,p_2}^{\PP^{n_2}}$.

\begin{theorem}
    $Y = \PP^{n_1}\times \PP^{n_2}$ such that $1\leq n_1<n_2$ and $d_1 \geq 1$, $d_2 \geq 2$, then $B(q)$ and $D(q)$ hold for every $q$ unless $n_1=1$, $n_2$ is odd and $d_2=2$.
\end{theorem}

\begin{proof}
    Suppose we have a non-zero Bott summand $b_{q_1,p_1}^{\PP^{n_1}} \cdot b_{q_2,p_2}^{\PP^{n_2}}$ in either $\mathbf{B}(q)$ or $\mathbf{D}(q)$ which means
    \begin{equation}\label{equation:balance}
        p_1 + p_2 = n_1 + n_2 - (q_1+q_2) - \varepsilon
    \end{equation}
    with $\varepsilon \in \{1,2\}$.

    We prove that each option in Table \ref{table:PnxPm} delivers a contradiction, column by column.
    We start with the column with $\Low$ in second place.
    Recall that $\Low = \Low' \sqcup \Low^0$, so we treat these cases separately.
    
    \textbf{First column: 2 $\in\Low'$}.

    That is, we address the cases of non-vanishing $(\Low_{n_1},\Low_{n_2}),(\Mid_{n_1},\Low_{n_2}),(\Up_{n_1},\Low_{n_2})$.
    This means, by definition, that $H^{q_2}(\PP^{n_2}, \Omega_{\PP^{n_2}}^{p_2}(\ell_2)) \neq 0$ with $q_2 = 0$ and $\ell_2 \geq p_2+1$.

    Thus
    \[p_2 \leq n_2 +1 - d_2(q+1) -1 \leq n_2 - 2q -2 \enspace .\]
    Then combining with (\ref{equation:balance}) gives $p_1 = n_1 + n_2 -q - \varepsilon - p_2 \geq n_1 + q +2 - \varepsilon$.
    Since $p_1 \leq n_1$ this rearranges to give $ \varepsilon -2 \geq q$ which forces $q=0$, $\varepsilon = 2$ and hence $p_1 = n_1$.
    Then $q = q_1 = 0$ and thus $1 \in \Low'$ which means that $\ell_1 \geq p_1+1 = n_1+1$ and hence $d_1 \leq 0$ which is a contradiction.

    \textbf{First column: 2 $\in \Low^0$}.

    The $\Low^0$ conditions say that $q_2 = p_2 = \ell_2 = 0$, $q=q_1$.
    So $\ell_2=0$ and $d_2 \geq 2$ means that 
    \begin{equation}\label{equation:n2-geq-2q+1}
        n_2+1 = d_2(q+1) \geq 2q+2 \text{ with equality if and only if } d_2 =2 \enspace .
    \end{equation}

    Now using (\ref{equation:n2-geq-2q+1}), (\ref{equation:balance}) and that $q_1,p_1 \leq n_1$ we get $2n_1 \geq q_1+p_1 = q+p = n_1+n_2 - \varepsilon$ giving
    \begin{equation}\label{equation:n_2-leq-n_1+epsilon}
        n_2 \leq n_1 + \varepsilon.
    \end{equation}

    Now we need to split into the different cases for the first summand.
    \begin{enumerate}
        \item Firstly assume that $1 \in \Low$. 
        Then $q=q_1=0$ and $p_1 = n_1 + n_2 - \varepsilon$ by (\ref{equation:balance}).
        Since $n_1 \geq p_1$, this forces $n_2=2$ and then (\ref{equation:n_2-leq-n_1+epsilon}) forces $n_1=p_1=1$ and $\varepsilon=2$.
        Thus $1 \in \Low'$ and so $\ell_1 = 2 - d_1 \geq p_1 + 1=2$ and so $d_1\leq 0$ a contradiction.

        \item Now suppose $1 \in \Mid$.
        Thus $\ell_1=0$ and $p_1=q_1 = q$.
        Thus by (\ref{equation:balance}) and (\ref{equation:n2-geq-2q+1}), we have $2q = n_1+n_2-\varepsilon \geq n_1+2q+1-\varepsilon$.
        Hence $n_1=1$, but then $1 \in \Mid$ means $1\leq q_1 \leq n_1-1 = 0$ this is impossible.

        \item Lastly, let $1 \in \Up$. 
        Thus $q=q_1=n_1$ and hence (\ref{equation:balance}) means $p_1=n_1+n_2-q_1-\varepsilon=n_2-\varepsilon$.
        Then (\ref{equation:n2-geq-2q+1}) and (\ref{equation:n_2-leq-n_1+epsilon}) mean $2n_1 +1 \leq n_1+\varepsilon$ and so $n_1 = 1$ and $\varepsilon=2$.
        Further $n_2=3$ and $d_2 = \frac{n_2+1}{q+1} = 2$ which is precisely the excluded case.
    \end{enumerate}

    \textbf{Second column: 2 $\in \Mid$}.

    Since $\ell_2 =0$, we have (\ref{equation:n2-geq-2q+1}) again.
    Again we split into the cases for the first summand.

    \begin{enumerate}
        \item Let $1 \in \Low'$. 
        Then $q=q_2$ and $p_1 \leq \ell_1 - 1 = n_1 - d_1(q+1) \leq n_1 - q-1$.
        Now (\ref{equation:balance}) means $p_1 = n_1+n_2-\varepsilon - 2q_2 \leq n_1-q-1$.
        Thus $n_2 \leq q+\varepsilon -1$.
        Combining this with (\ref{equation:n2-geq-2q+1}) gives $2q+1 \leq q+\varepsilon -1$ forcing $q \leq 0$ a contradiction.

        \item Now if $1 \in \Low^0$ again $q = q_2$ and also $\ell_1=p_1=0$.
        So (\ref{equation:balance}) gives $0 = n_1 + n_2 - \varepsilon -2q \geq n_1+1 -\varepsilon$ and so $\varepsilon -1 \geq n_1$ giving $n_1 =1$ and $\varepsilon = 2$.
        So we see that $n_2 = 2q+1$ and (\ref{equation:n2-geq-2q+1}) says that $d_2=2$.
        Moreover, $\ell_1=0$ and so $(q+1)d_1=n_1+1 = 2$ and since $q \geq 1$ we have $q=1,d_1=1$ and $n_2 = 3$, exactly an exceptional case.

        \item Now let $1\in \Mid$. 
        So then $\ell_1=\ell_2=0$, so summing them together we get $(q+1)(d_1+d_2) = n_1+n_2+2$.
        Then (\ref{equation:balance}) means $n_1+n_2 = 2q+\varepsilon$ and combining these we get $3q+3 \leq 2q+2+\varepsilon$ and so $q \leq 1$.
        However, $1,2 \in \Mid$ means $q = q_1+q_2 \geq 2$, a contradiction.

        \item Let $1 \in \Up^0$.
        Then $\ell_1=0$, that is, $(q+1)d_1 = n_1+1$ and $q \geq q_1=n_1$.
        This can only happen when $d_1=1$.
        But then $q=n_1=q_1$ which forces $q_2=0$ a contradiction.

        \item Finally let $1 \in \Up'$.
        Thus $q_1=n_1$ and write $q_2 = q - n_1$.
        Then (\ref{equation:balance}) says
        \begin{align*}
            (q_1+p_1)+(q_2+p_2) &= (n_1+p_1) + (2q-2n_1)\\
            &=n_1+n_2-\varepsilon
        \end{align*}
        and so $p_1 = 2n_1+n_2-\varepsilon -2q$.

        If $d_2 \geq 3$ then $(n_2+1) =d_2(q+1) \geq 3q+3$ and thus $n_2-2q \geq q+2$.
        Hence $p_1 \geq 2n_1 + q + 2-\varepsilon > n_1$ a contradiction.
        Thus $d_2=2$ and so $n_2=2q+1$, hence odd, and $p_1 = 2n_1+1-\varepsilon \leq n_1$.
        This forces $n_1=1,\varepsilon =2$ and $p_1=1$ and we are in an exceptional case.
    \end{enumerate}
    \textbf{Third column: $2\in \Up$}.

    Now $q_2=n_2$ which means $q=q_1+q_2 \geq n_2$.
    Firstly, if $2 \in \Up^0$ then $\ell_2=0$ forces $d_2 \leq 1$, a contradiction.
    
    So assume $2 \in \Up'$.
    If $1 \in \Low^0 \cup \Mid \cup \Up^0$, then $\ell_1=0$ forces $d_1<1$ which is not possible and we go through the remaining cases.
    
    \begin{enumerate}
        \item Let $1 \in \Low'$.
        Then $\ell_1 \geq p_1+1\geq 1$.
        But again $(q+1)d_1 \geq q+1 \geq n_2+1>n_1$ a contradiction.

        \item Finally if $1 \in \Up'$ then $q=n_1+n_2$ so (\ref{equation:balance}) means $p_1+p_2 = -\varepsilon$ a contradiction.
    \end{enumerate}
\end{proof}

\begin{remark}
    We can compute explicitly the dimensions in the cases where our argument breaks down.
    
    For example, if $\PP^2 \times \PP^3$ and $d_1=d_2=1$, then $\mathbf{B}(0)=12$.
    
        From the proof we see that when $n_2$ is odd and $d_2=2$ then $\mathbf{B}(\frac{n_2-1}{2})$ fails.
    Indeed, let $n_2 = 2m+1$ and $m>1$ then
    \[h^1(\PP^1,\Omega^{1}_{\PP^1}(\ell_1)) \cdot h^{q-1}(\PP^{n_2},\Omega^{q-1}_{\PP^{n_2}}) = (q+1)d_1 -1 \neq 0 \enspace .\]

    However, that does not necessarily mean that the automorphism groups are positive dimensional, only that the method fails, see Section \ref{section:pencils-of-quadrics}.
\end{remark}

Now we proceed with the other general result where no factors of $\PP^1$ are allowed.

\begin{theorem}\label{theorem:general-products}
    Let $Y = \prod\PP^{n_j}$ and $n_j,d_j\geq 2$.
    Then $\mathbf{B}(q)= \mathbf{D}(q) = 0$ for all $q$.
\end{theorem}

Note that for each \(j\), we have \(p_j\leq n_j\).
For a fixed non-zero Bott summand, we set
\[
N_{\Low}:=\sum_{j\in\Low}n_j,\qquad
N_{\Mid}:=\sum_{j\in\Mid}n_j,\qquad
N_{\Up}:=\sum_{j\in\Up}n_j,
\]
and write \(N=N_{\Low}+N_{\Mid}+N_{\Up}\). Then
\[
q=\sum_{j=1}^{r}q_j,\qquad
P=\sum_{j=1}^{r}p_j=N-q-\varepsilon,
\]
where \(\varepsilon\in\{1,2\}\).

We begin by showing that $\Up^0$ is empty.
\begin{lemma}
If \(d_{j} \geq 2,\)  then \(\Up^{0}=\emptyset.\) 
\end{lemma}
\begin{proof}
 Assume \(j\in \Up^{0},\) then by definition \(\ell_{j}=0\) and \(q_{j} = n_{j},\) which implies \begin{align*}
  n_{j}+1&=(q+1)d_{j}\\
 \Rightarrow d_{j}&=\frac{n_{j}+1}{q+1}
 \end{align*}
 Since \(q=\sum\limits_{j}q_{j}= n_{j} +\sum\limits_{j^{'}\neq j}q_{j^{'}}, \) then \(d_{j}\leq \frac{n_{j}+1}{n_{j} +\sum\limits_{j^{'}\neq j}q_{j^{'}}+1}\leq 1.\) This is a contradiction since \(d_{j} \geq 2\).
\end{proof}

The following result gives a technical and useful inequality.
\begin{lemma}\label{inequalityofN} 
The inequality  \(N\leq 2q +\varepsilon +\sum\limits_{j\in \Low^{'}}(\ell_{j}-1)\) holds.
\end{lemma}

\begin{proof}
Notice that \(N\leq 2q +\varepsilon +\sum\limits_{j\in \Low^{'}}(\ell_{j}-1)\) means \(N-q-\varepsilon \leq q+ \sum\limits_{j\in \Low^{'}}(\ell_{j}-1).\) Moreover,
\begin{align*}
N-q-\varepsilon &=\sum\limits_{j\in \Low\cup \Mid\cup \Up} p_{j}\\
&=\sum\limits_{j\in  \Low} p_{j}+\sum\limits_{j\in  \Mid} p_{j}+\sum\limits_{j\in \Up} p_{j}\\
&\leq \sum\limits_{j\in  \Low^{'}}( \ell_{j} -1)+\sum\limits_{j\in  \Mid} q_{j}+\sum\limits_{j\in \Up^{'}} n_{j}\tag{since in \(\Low^{0},p_{j}=0\) and \(\Up^{0}=\emptyset \)}\\
 &=\sum\limits_{j\in \Low^{'}}( \ell_{j} -1)+\sum\limits_{j\in  \Mid} q_{j} +N_{\Up}\\
 &=\sum\limits_{j\in \Low^{'}}( \ell_{j} -1)+q\tag{ since  \( q_{j}=0\) in \(\Low^{'}\) and \(n_{j}=q_{j} \) in \(\Up\)}\\
 \text{so,}~ N-q-\varepsilon &\leq q+ \sum\limits_{j\in \Low^{'}}(\ell_{j}-1) 
\end{align*}
\end{proof}

\begin{lemma}\label{Low:nonempty}
If \(q=0, \Mid=\Up=\emptyset\) then \(\mathbf{D}(q)=\mathbf{B}(q)=0.\)
\end{lemma}
\begin{proof}
For \(q=0\), Lemma \ref{inequalityofN} gives \(N-\varepsilon \leq  \sum\limits_{j\in \Low^{'}}(\ell_{j}-1).\) Since \(j\in \Low^{'},\) then \(\ell_{j}= n_{j}+1-d_{j}, \) which implies
\begin{align*}
\ell_{j}-1&=n_{j}-d_{j}\leq n_{j}-2\tag{\(d_j\geq 2\)}\\
\Rightarrow N-\varepsilon &\leq \sum\limits_{j\in \Low^{'}}(n_{j}- 2)=N_{\Low^{'}}-2|\Low^{'}|\\
\end{align*}
Hence \( N-N_{\Low^{'}}+2|\Low^{'}| \leq \varepsilon \leq 2,\) since \(\varepsilon=1,2\). This contradicts \(r\geq 2.\) 
\end{proof}
\begin{lemma}\label{Mid:nonempty}
If \(q\geq 1, \Low=\Up=\emptyset,\) it follows that \(\mathbf{D}(q)=\mathbf{B}(q)=0.\)  
\end{lemma}
\begin{proof}
If  \(\Low=\Up=\emptyset,\) then \(q=\sum\limits_{j}q_{j}=\sum\limits_{j}p_{j}=N-q-\varepsilon.\)  Thus, \(q=\frac{N-\varepsilon}{2}=\frac{N_{\Mid}-\varepsilon}{2}.\)  Notice also that \(\ell_{j}=0,\) so if for example \(\ell_{1}=0,\) it implies that 
\[d_{1}=\frac{n_{1}+1}{q+1}=\frac{n_{1}+1}{\frac{N_{\Mid}-\varepsilon}{2}+1}=\frac{2(n_{1}+1)}{N_{\Mid}-(\varepsilon-2)}< 2 \tag{\(\varepsilon=1,2\)}\]
This is a contradiction since each of the \(d_{j}\)'s is at least 2.
\end{proof}

\begin{lemma}\label{c:nonempty}
   For all \(j,\) if \(\Low=\Mid=\emptyset\)  then  \(\mathbf{D}(q)=\mathbf{B}(q)=0.\)
\end{lemma}
\begin{proof}
If \(\Low=\Mid=\emptyset\),  then all \(j\) are in \(\Up,\)  \(q_{j}=n_{j}\) and \(q=\sum\limits_{j}n_{j}=N\). Thus, \(p=N-q-\varepsilon=q-q-\varepsilon=-\varepsilon,\) which is a contradiction since \(p\geq 0.\)
\end{proof}
\begin{remark}\label{cor.inequality}
The following inequality will be very useful: \[N_{\Low^{0}}+N_{\Mid}+N_{\Up} \leq 2q+\varepsilon-2(q+1)|\Low^{'}|.\]
\end{remark}
\begin{proof}
Recall that for all \(j\), \(\ell_{j}=n_{j}+1-(q+1)d_{j}, \) which implies that \begin{align*}
   &~  \ell_{j}-1 \leq n_{j}-2(q+1)\tag{\(d_{j}\geq 2\)}\\
&~ \Rightarrow \sum\limits_{j\in \Low^{'}}(\ell_{j}-1) \leq N_{\Low^{'}} -2(q+1)|\Low^{'}| \tag{summing over \(j\in \Low^{'}\)}\\
&~ \Rightarrow N \leq 2q +\varepsilon + N_{\Low^{'}}- 2(q+1)|\Low^{'}|\tag{by Lemma \ref{inequalityofN}}\\
&~ \Rightarrow N_{\Low^{0}} +N_{\Low^{'}}+N_{\Mid}+N_{\Up} \leq 2q+N_{\Low^{'}}+\varepsilon-2(q+1)|\Low^{'}|\\
&~ \Rightarrow N_{\Low^{0}} +N_{\Mid}+N_{\Up} \leq 2q+\varepsilon-2(q+1)|\Low^{'}|
\end{align*}
\end{proof}
\begin{lemma}\label{all:nonempty}
        If \(\Mid\neq \emptyset,\) then \(\mathbf{D}(q)=\mathbf{B}(q)=0.\) 
\end{lemma}
\begin{proof}
If \(\Mid\neq \emptyset,\) then for all \(j\in \Mid,\ell_{j}=0,\) which implies \(n_{j}+1=(q+1)d_{j}\geq 2(q+1),\) hence \(n_{j}\geq 2q+1\) and moreover \(N_{\Mid} \geq 2q+1. \) By Remark \ref{cor.inequality},
\begin{align*}
&2q+1 +N_{\Low^{0}}+N_{\Up} \leq2q+\varepsilon-2(q+1)|\Low^{'}|\\
&\Rightarrow 0\leq N_{\Low^{0}}+N_{\Up} \leq \varepsilon-2(q+1)|\Low^{'}|-1\\
  &\Rightarrow 2(q+1)|\Low^{'}| \leq \varepsilon-1\leq 1 \tag{\(\varepsilon=1,2\)}\\
 & \Rightarrow |\Low^{'}|=0 \Leftrightarrow  \Low^{'}=\emptyset
\end{align*}
Now, by Lemma \ref{inequalityofN} \(N\leq 2q+\varepsilon,\) and also for all \(j\in \Mid, n_{j}\geq 2q+1,\)  which implies that 
\begin{align*}
&N\geq n_{j}+2 \geq 2q+3 \tag{\(r\geq 2\) and  for all \(j, n_{j}\geq 2\)}\\
&\Rightarrow 2q+\varepsilon \geq 2q +3,
\end{align*}
which is a contradiction since \(\varepsilon=1,2.\)
\end{proof}
\begin{lemma}\label{LandU:nonempty}
Let \(q\geq 2, \Mid=\emptyset\) and \(\Up\neq \emptyset \) then \(\mathbf{B}(q)=\mathbf{D}(q)=0\).
\end{lemma}
\begin{proof}
If \(\Mid=\emptyset,\) then \(q= q_{1}+\cdots +q_{r}=\sum\limits_{j^{'}\in \Low} q_{j^{'}}+ \sum\limits_{j\in \Up}q_{j} =N_{\Up},\) since \(q_{j^{'}}=0,\) for all \(j^{'} \in \Low.\) Applying Remark \ref{cor.inequality} gives 
\begin{align*}
&N_{\Low^{0}}+N_{\Mid}+N_{\Up} \leq 2q+\varepsilon-2(q+1)|\Low^{'}|\\
&\Rightarrow N_{\Low^{0}}+N_{\Up} \leq 2N_{\Up}+\varepsilon-2(q+1)|\Low^{'}|\tag{\(q=N_{\Up}\) and \(N_{\Mid}=0\)}\\
&\Rightarrow N_{\Low^{0}} \leq N_{\Up}+\varepsilon-2(q+1)|\Low^{'}|\\
&\Rightarrow N_{\Low^{0}} \leq q+\varepsilon-2(q+1)|\Low^{'}| \tag{\(q=N_{\Up}\)}\\
\end{align*}
From this we check two cases; either \(|\Low^{'}|\geq 1\) or \(|\Low^{'}|=0\).
If \(|\Low^{'}|\geq 1,\) then \(N_{\Low^{0}} \leq q+\varepsilon-2(q+1)|\Low^{'}\leq -1\), which is a contradiction.

If  \(|\Low^{'}|=0\) by Remark \ref{cor.inequality}, we have that \(N_{\Low^{0}}+N_{\Up} \leq 2N_{\Up} +\varepsilon,\)   which implies that \(N_{\Low^{0}}\leq q+\varepsilon.\) Recall that for all \(j\in \Low^{0},~ \ell_{j}=0 \) which implies  that \(n_{j}\geq 2q+1\) and thus \(N_{\Low^{0}}\geq 2q+1.\) This means that \(2q+1 \leq N_{\Low^{0}}  \leq q+\varepsilon \), since $q\ge 2$ this means $\Low^0 = \emptyset$ and we conclude by Lemma \ref{c:nonempty}.
\end{proof}

We combine all the above lemmas to prove our result.
\begin{proof}[Proof of Theorem \ref{theorem:general-products}]
If one of  \(\mathbf{B}(q)\) or \(\mathbf{D}(q)\) do not hold, choose a non-zero Bott summand. Then applying the Bott vanishing criterion in Theorem \ref{theorem:Bott-vanishing-Pn}, means that for all \(j\in \{1,\dots,r\}=\Low\cup \Mid \cup \Up,\) the products \(\prod b_{q_{j},p_{j},\ell_{q}}\) are all non zero.
This phenomenon breaks up in the following ways;
\begin{enumerate}
\item[1.] If \(\Mid=\Up=\emptyset,\) then all \(j\in \Low\) and this contradiction is Lemma \ref{Low:nonempty}.
\item[2.] If \(\Low=\Up=\emptyset,\) then all \(j\in \Mid\) and this contradiction is Lemma \ref{Mid:nonempty}.
\item[3.]If \(\Low=\Mid=\emptyset,\) then all \(j\in \Up\) and this contradiction is Lemma \ref{c:nonempty}.
\item[4.] If \(\Mid\neq \emptyset, \) then all \(j\) are  in \(\Mid\)  or \(\Low\cup \Mid\) or \(\Mid\cup \Up\) or \(\Low\cup \Mid\cup \Up\). Moreover, this also implies that \(q\geq 1\) which is the contradiction in  either Lemma \ref{Mid:nonempty} or Lemma \ref{all:nonempty}.
\item[5.] If \(\Mid=\emptyset\) and \(\Up\neq \emptyset,\) with \(q\geq 2\) then  either all \(j\) are in \(\Up\) (which is covered in Lemma \ref{c:nonempty}) or \(\Low\cup \Up\) which is implied by Lemma \ref{LandU:nonempty}.
\end{enumerate}

\end{proof}

\subsection{Exceptional families and failure of the vanishing criterion}

Here we examine cases where the above method does not work.
That is, the vanishing \eqref{Bq} or \eqref{Dq} fails.
These cases do not necessarily provide counterexamples.
There are three possibilities:

\begin{enumerate}
    \item $\mathbf{C}(0)$ holds for every smooth hypersurface;
    \item the generic hypersurface has $\mathbf{C}(0)$, but there is a closed locus where $h^0(T_X)>0$;
    \item $h^0(T_X) >0$ for every smooth hypersurface in the linear system.
\end{enumerate}

We shall see that all 3 possibilities are realised.

\subsubsection{Pencils of even-dimensional quadrics}\label{section:pencils-of-quadrics}
    Assume that $\Char k \neq 2$.
    We consider a hypersurface $X$ of degree $(d,2)$ with $d>1$ in $Y=\PP^1 \times \PP^{n}$ for $n>1$ and odd.
    We take coordinates $([u:v],[x])$ on $Y$ and suppose that $F(u,v;x)$ is a bihomogeneous form defining our smooth hypersurface $X$.
    Then we can write $F= x^TQ(u,v)x$ where $Q$ is a symmetric $(n+1)\times (n+1)$ matrix whose entries are binary forms of degree $d$.
    Consider $\pr_1: X \to \PP^1$ which is a pencil of even-dimensional quadrics.
    The discriminant of this family is $\Delta(X):= \det Q(u,v)$.
    Then, classically \cite[Proposition 2.1]{Reid:1972} we know that $X$ is smooth if and only if $\Delta(X)$ is square-free.

    The failure of $\mathbf{B}(\frac{n-1}{2})$ is exactly given by the primitive middle cohomology of an even-dimensional quadric \cite[Proposition 1.12]{Reid:1972}.

    However, the automorphism group is still finite and hence $\mathbf{C}(0)$ holds.
    This follows since any automorphism of $X$ would be an automorphism of $Q$ and thus give an automorphism of $\PP^1$ preserving the discriminant which is at least 4 points.
    Thus the group must be finite.

    If we consider a hypersurface $X = V(uQ_0(x)+vQ_1(x))$ of degree $(1,2)$, we have $X = \Bl_Z\PP^n$ where $Z = V(Q_0,Q_1)$.
    Again, using arguments with the discriminant, one can prove that for $X$ smooth, the automorphism group must be finite.

\subsubsection{Syzygy bundles}

The other, more interesting class of exceptional families where our method fails is where the hypersurfaces are projectivisations of so-called \emph{syzygy bundles}.

Write
\[ F(x,y)=\sum_{j=0}^{n_2}y_jF_j(x), \qquad F_j\in H^0(\PP^{n_1},\cO_{\PP^{n_1}}(d)). \]
The $F_j$ define a sequence
\[
0\longrightarrow \cK\longrightarrow
\cO_{\PP^{n_1}}^{\oplus n_2+1}
\xrightarrow{(F_0,\ldots,F_{n_2})}
\cO_{\PP^{n_1}}(d)\longrightarrow0.
\]
Provided that the \(F_j\) have no common zero, this sequence is exact on the right and \(\cK\) is a vector bundle of rank \(n_2\).  
With the convention that $\PP(\cK)$ parametrises lines in its fibres, the fibre over $x\in\PP^{n_1}$ is
\[
 \{y\in\PP^{n_2}:\textstyle\sum_jy_jF_j(x)=0\}
   =\PP(\cK)_x \enspace .
\]
Consequently
\[
 X\simeq\PP_{\PP^{n_1}}(\cK).
\]

If the $F_j$ are linearly independent and $W=\langle F_0,\ldots,F_{n_2}\rangle$, then \(\cK\) is a syzygy bundle $M_W$ in the sense of \cite[Introduction]{Coanda:2011}.

One can compute using the Bott formulas above that indeed the vanishing required by our method fails specifically $\mathbf{B}(2)$.
However, there are bidegrees where finiteness holds generically, but smooth members exist with $H^0(X,T_X) \neq 0$ and hence positive dimensional automorphism group scheme, at least in characteristic 0.
We illustrate this with one concrete family.

Assume that $k$ is algebraically closed of characteristic zero.
Let $W\subset H^0(\PP^2,\mathcal O(2))$ be a basepoint-free four-dimensional subspace, with basis $F_0,\ldots,F_3$, and define its syzygy bundle by

\[
0\longrightarrow M_W\longrightarrow W\otimes\mathcal O_{\PP^2}
\longrightarrow\mathcal O_{\PP^2}(2)\longrightarrow0.
\]

With the convention that projectivisation parametrises lines,
\[
X_W:=V\Bigl(\sum_{i=0}^3y_iF_i(x)\Bigr)
\simeq\PP_{\PP^2}(M_W)
\subset\PP^2\times\PP^3.
\]
In particular, $X_W$ is smooth since it is the projective bundle $\PP(M_W)$.
The vanishing required by our method fails, since Bott's formula gives
\[H^2(\PP^2,\Omega_{\PP^2}^1(-3)) \otimes H^0(\PP^3,\mathcal O_{\PP^3}(1)) \neq0.\]

Nevertheless, $\Aut(X_W)$ is finite for general $W$.
Indeed, by Grothendieck-Lefschetz we have that $\Pic(X_W) \cong \ZZ \oplus \ZZ$.
Automorphisms cannot exchange the two copies of $\ZZ$, so their complete linear systems extend every automorphism to $\PP^2\times\PP^3$.
Linear independence of the $F_i$ then gives
\[\Aut(X_W)\simeq\Stab_{\PGL_3}(W).\]
Passing to the annihilator of $W$ identifies this stabiliser with that of a pencil of quadrics in the dual representation.
For general \(W\), the pencil \(\PP(W^\perp)\) has four distinct base points in general position, so its stabiliser acts faithfully on this four-point set and is therefore finite; since \(\Stab_{\PGL_3}(W)=\Stab_{\PGL_3}(W^\perp)\), the same holds for \(\Stab_{\PGL_3}(W)\).

On the other hand, consider
\[X_0=V(y_0x_0^2+y_1x_1^2+y_2x_2^2+y_3x_0x_1).\]
Its coefficients have no common zero, so $X_0$ is smooth.
The diagonal torus of $\PP^2$ acts faithfully on $X_0$, with compensating weights on the $y_i$.
Thus $h^0(X_0,T_{X_0})>0$, although the general member of this family has finite automorphism group.


\section{Applications to moduli spaces}

In this section we describe an immediate application to moduli spaces of hypersurfaces. Fix $Y$ a projective simplicial toric variety and we additionally assume that $\Aut(Y)$ is (geometrically) reductive.

Note that this is a strong hypothesis which is not always satisfied. Indeed toric varieties have automorphism groups which are non-reductive.
Also note that, in positive characteristic, when the automorphism groups identity component is not a torus, it is not linearly reductive.
However, there is a large class of toric varieties which do in fact have reductive automorphism group, even many weighted projective spaces.

Fix a quasismooth ample divisor \(X\subset Y\), and let \(G=\Aut(Y,[X]) \subset \Aut(Y)\) be the group of automorphisms of $Y$ which fix the class \([X] \in \Cl(Y)\).
This is a finite index subgroup of \(\Aut(Y)\).

Denote the complete linear system of \([X]\) system by
\[H=\lvert\cO_Y(X)\rvert .\]
We assume that
\[H^0(X',T_{X'})=0\]
for every quasismooth \(X'\in H\).

\begin{definition}
    Let \(\Delta\) be the regular $A$-determinant associated to the linear system \(|X|\) as defined in \cite[Chapter 11]{GKZ}.
    We say that \(H\) is non-defect if \(\Delta\) is non-constant.
    In this case we have 
    \[H^{\QS}=H_\Delta ,\]
    as in \cite[Theorem 11.1.6]{GKZ}.
\end{definition}

We shall assume from now on that $H$ is non-defect.

\begin{remark}
In characteristic zero the nondefect hypothesis is very mild.
If \(P_X\) is the lattice polytope associated to \(\cO_Y(X)\), then it holds whenever \(P_X\) has lattice width at least \(2\).
Indeed, a dual-defective toric configuration necessarily has lattice width one \cite{DickensteinNill:2010,FurukawaIto:2021}.

In particular, if \(X\sim mD\) with $D$ an ample Weil divisor whose monomial linear system has full-dimensional support and \(m\geq2\), then \(P_X=mP_D\) has lattice width at least \(2\), and the hypothesis is automatic.
Thus Theorem~\ref{theorem:moduli-stacks} automatically applies to every second or higher multiple of an ample Cartier toric polarization.
\end{remark}

We denote the quotient stack by
\[
\cM_{X\subset Y}^{\QS}
   :=\left[H^{\QS}/G\right].
\]

Using our results on automorphism groups, we obtain a toric analogue of Benoist's result for hypersurfaces in projective space \cite{Benoist:2013}.

\begin{theorem}\label{theorem:moduli-stacks}
The stack \(\cM^{\QS}_{X\subset Y}\) is a smooth separated Deligne--Mumford stack with coarse moduli space
\[\pi\colon \cM^{\QS}_{X\subset Y}\longrightarrow M,\]
where \(M\) is a normal affine variety.
\end{theorem}

\begin{proof}
We may assume that \(k\) is algebraically closed and note \(H^{\QS}\) is a smooth affine variety.

For every \(X'\in H^{\QS}\), we have that \(\Stab_G(X')\subseteq\Aut(X')\).

Our hypothesis \(H^0(X',T_{X'})=0\) means that \(\Stab_G(X')\) is finite and reduced.
Hence the quotient stack is DM.
Moreover, all \(G\)-orbits in \(H^{\QS}\) have dimension \(\dim G\).
Thus every orbit in \(H^{\QS}\) is closed.

It follows that every point of \(H^{\QS}\) is stable.
Hence the action of \(G\) on \(H^{\QS}\) is proper by \cite[Corollary~2.5]{GIT}.
Thus the quotient stack is separated; its finite reduced stabilisers show that it is Deligne--Mumford. It is smooth because both \(H^{\QS}\) and \(G\) are smooth.

Finally, the GIT quotient
\[M=H^{\QS}/\!/G = H^\QS / G\]
is the coarse moduli space.
Since \(H^{\QS}\) is affine and normal and \(G\) is geometrically reductive, \(M\) is a normal affine variety associated to the ring of invariants.
\end{proof}

\begin{remark}[Further work]
It would be natural to extend this construction to quasismooth complete intersections in toric varieties, in the spirit of \cite{Benoist:2013}.
Beauville's recent work on maximal variation of linear systems \cite{Beauville:2026} suggests a deformation-theoretic perspective on these quotient stacks; understanding this connection is work in progress.

Further, it should be possible using the theory of non-reductive quotients to drop the geometrically reductive assumption in some cases.
This is also work in progress.
\end{remark}

\printbibliography

\bigskip

\noindent
\textsc{Dominic Bunnett}\\
Institut für Mathematik, Technische Universität Berlin, Berlin, Germany\\
\texttt{bunnett@math.tu-berlin.de}

\medskip

\noindent
\textsc{Caroline Namanya}\\
Department of Mathematics, Makerere University, Kampala, Uganda\\
\texttt{caronamanya97@gmail.com}

\end{document}